\documentclass[10pt]{article}

\usepackage[margin = 1.0in]{geometry}   
\usepackage{geometry}  			
\usepackage{amssymb}
\usepackage{amsmath}
\usepackage{palatino}
\usepackage{graphicx}
\usepackage[english]{babel}
\usepackage{amscd}
\usepackage{amsthm}
\usepackage{color}
\usepackage{enumerate}
\usepackage{hyperref}
\usepackage{comment}
\theoremstyle{plain}
\newtheorem{theorem}{Theorem}[section]                                          
                          
\newtheorem{lemma}[theorem]{Lemma}
\newtheorem{corollary}[theorem]{Corollary}

\theoremstyle{definition}

\theoremstyle{remark}

\newcommand{\Prb}{\mathbb{P}}
\newcommand{\Exp}{\mathbb{E}}

\title{On lower bounds for hypergeometric tails}

\author{
Jianhang Ai\thanks{Faculty of Electrical Engineering, Czech Technical University, Karlovo N\'am\v{e}st\'i 13, 12135, Prague 2, Czech Republic, e-mail: ai.jianhang@fel.cvut.cz}
\and
Christos Pelekis\thanks{Aristotle University of Thessaloniki, 
Department of Mathematics, 
541 24 Thessaloniki, Greece, e-mail: pelekis.chr@gmail.com, \, cpelekis@math.auth.gr }
}

\begin{document}

\maketitle

\begin{abstract}  
Let $n,k$ be positive integers such that $n\geq k$, 
and let $H$ be a hypergeometric random variable counting the  number of black marbles in a sample without replacement of size $k$ from an urn that contains $i\in \{1,\ldots, n\}$
black and $n - i$ white marbles. 
It is shown that 
\[
\mathbb{P}(H \ge \mathbb{E}(H)) \ge k/n\, , \, \text{when} \,\,  n\ge 8k \, . 
\]
Furthermore, provided that $1\le \Exp(H)\le \min\{i,k\}-1$ as well as that $\frac{(n-i)(n-k)}{n}>1$, it is shown that  
\[
\mathbb{P}(H\ge \mathbb{E}(H)) \,\ge\, \frac{e^{-1/12}}{4\sqrt{2}} \cdot \sqrt{\frac{n-1}{n}} \cdot\frac{ \sqrt{\text{Var}(H)} }{1 + \sqrt{1+ \frac{n-1}{n-k}\cdot\text{Var}(H)}}\, .
\]
Auxiliary results which may be of independent interest include an upper bound on the tail conditional expectation and a  lower bound on the mean absolute deviation of the hypergeometric distribution. 

\end{abstract}



\section{Introduction, related work and main results} 

Throughout the text, given a positive integer $n$, we denote by $[n]$ the set $\{1,\ldots,n\}$. Given two random variables $X,Y$, we write  $X\sim Y$  to denote the fact that $X$ and $Y$ have the same distribution.
Given positive integers 
$n,k$ such that $n\ge k\ge 1$ and a positive integer $i\in [n]$, we denote by $\text{Hyp}(n,i,k)$ a hypergeometric random variable which counts the number of black marbles in a sample without replacement of size $k$ from an urn that contains $i$
black and $n - i$ white marbles. 
Formally, if $H\sim\text{Hyp}(n,i,k)$, we have   
\[
\Prb(H=j)=\frac{\binom{i}{j}\binom{n-i}{k-j}}{\binom{n}{k}},\quad \text{ for } \, 
j\in\{\max\{0,k-(n-i)\},\dots,\min\{k,i\}\} \, .
\]
Recall that $\Exp(H)= \frac{ik}{n}$ and $\text{Var}(H) = k\cdot\frac{i}{n}\cdot\frac{n-i}{n}\cdot\frac{n-k}{n-1}$. Finally, 
given a positive integer $n$ and a real $p\in [0,1]$, we denote by $\text{Bin}(n,p)$ a binomial random variable of parameters $n,p$.

In this article we shall be concerned with lower bounds on the probability that  a hypergeometric random variable exceeds its expected value. More precisely, let $H\sim \text{Hyp}(n,i,k)$. What is a sharp lower bound on the ``tail" $\Prb(H \ge ik/n)$? 

Such a question arises in several settings. 
For instance,  suppose that a smuggler has a certain quantity, say $1$, of illicit product that he needs to transport across the border. The smuggler employs $n$ agents, which he sends across the border simultaneously, knowing that exactly  
$k$ of the agents will be chosen uniformly at random and seized by the authorities.  
The smuggler decides to equally distribute the product over $i\in [n]$ agents; thus loading $i$ agents with an amount of $1/i$ of the product and $n-i$ agents with no amount. 
The operation remains profitable as long as he manages to keep the amount  not seized above the expected amount not seized.  Which choice of $i\in [n]$ maximizes the probability that the operation remains profitable? What is an upper bound on the probability that the operation remains profitable?

If $Y_i$ is the amount of the product that makes it across the border, then it holds $Y_i = 1 -\frac{1}{i} \cdot H_i$, for $H_i\sim\text{Hyp}(n,i,k)$, and the expected amount not seized is equal to $\Exp(Y_i)=1-\frac{1}{i}\cdot \frac{ik}{n} = \frac{n-k}{n}$. 
Hence,  the operation is profitable with probability   
\[
\Prb\left(Y_i >  \frac{n-k}{n}\right) = \Prb(H_i < ik/n) = 1 - \Prb(H_i \ge ik/n) \, .
\]  
In this article it is shown that, when $n\ge 8k$, it is optimal for the smuggler to load $n-1$ agents with no amount and one agent with the whole amount. 
 
Besides smuggling-versus-customs situations, the problem  arises in other settings, the most notable of which being the  \emph{Manickham-Mikl\'os-Singhi (MMS) conjecture}, which we briefly discuss here. Let $\mathcal{P}=\{x_1, \ldots,x_n\}$ be a population (a multiset) consisting of $n\ge 2$ real numbers whose sum is equal to zero. We sample $k\in [n]$ elements from $\mathcal{P}$ without replacement, and let $X_{\mathcal{P}}$ be the sum of the elements in our sample. The MMS conjecture asserts that, when $n\ge 4k$, it holds $\Prb(X_{\mathcal{P}}\ge 0)\ge \frac{k}{n}$. The bound is sharp, as can be seen by considering the population $\{1,-\frac{1}{n-1},\ldots,-\frac{1}{n-1}\}$. 
It is known (see Pokrovskiy~\cite{Pokrovskiy}) that the MMS conjecture holds true when $n\ge 10^{46}k$, and it has been conjectured by Aydinian and Blinovsky (see~\cite{Aydinian_Blinovsky}) that for every population $\mathcal{P}$ as above, and every $1\le k\le n$, it holds 
\[
\Prb(X_{\mathcal{P}}\ge 0) \ge \min_{i\in [n-1]} \, \Prb(H_i \ge ik/n) \, , \, \text{ where } \, H_i \sim\text{Hyp}(n,i,k) \, .
\]
Hence, provided that the Aydinian-Blinovsky conjecture holds true, the tails of  the hypergeometric distribution appear as extreme instances of the MMS conjecture, where each $H_i$ corresponds to a sample without replacement from the population $\mathcal{P}_i=\{\frac{1}{i},\ldots,\frac{1}{i}, \frac{-1}{n-i},\ldots,\frac{-1}{n-i}\}$ consisting of $i$ values that are equal to $\frac{1}{i}$, and $n-i$ values that are equal to $\frac{-1}{n-i}$. 

We also investigate the problem of estimating from below the probability that a hypergeometric random variable takes a value that is larger than or equal to its expectation. 
A closely related problem is to estimate from below the probability of a binomial random variable being larger than or equal to its expectation; a problem that has attracted considerable attention (see, for example,~\cite{Doerr, greenberg2014tight, Pelekis_Ramon, Pinelis}, and references therein).  For example, it is known (see Theorem~\ref{Greenberg_Mohri} below) that $\Prb(\text{Bin}(k,p)\ge kp)\ge 1/4$, when $kp>1$. 
However, the literature regarding the hypergeometric tail appears to be sporadic.  

There are, of course,  good reasons for this ``sporadicness".  By a result of Vatutin and Mikhailov (\cite{Vatutin_Mikhailov}, see also~\cite{Hui_Park}) it is known that a hypergeometric random variable is a sum of independent Bernoulli random variables. In other words, the hypergeometric distribution is a Poisson-binomial distribution. 
Hence, Berry-Esseen-type theorems for non-identically distributed independent summands apply. In particular, it is known (see~\cite[Formula~(17)]{Mattner_Schulz} for a more general result)  that, if
$H\sim\text{Hyp}(n,i,k)$,  it holds
\begin{equation}\label{Berry_Essen_bound}
\Prb(H\ge \Exp(H)) \ge \frac{1}{2} - \frac{0.5583}{\sqrt{\text{Var}(H)}} \, .
\end{equation}
It follows that $\Prb(H\ge \Exp(H))\ge 1/4$, when $\text{Var}(H)\ge (4\cdot0.5583)^2 \ge 4.98$. 
Note that the inequality~\eqref{Berry_Essen_bound} provides a non-trivial bound when $\text{Var}(H) \ge (2\cdot0.5583)^2 \ge 1.24$. 
Note also that while inequality~\eqref{Berry_Essen_bound} is simple, its proof is not, and it is natural to wonder whether one can obtain such a lower bound using ``first principles" of the hypergeometric law.

Moreover, by combining the above-mentioned result of  Vatutin and Mikhailov  with upper bounds on the total variation distance between the binomial and a Poisson-binomial distribution, 
it has been shown by Ehm (see~\cite[Theorem~2]{Ehm} for a more general result)  that,  when  $X\sim\text{Bin}(k,\frac{i}{n})$ is such that  $\text{Var}(X)\ge 1$,  it holds 
\begin{equation}\label{Ehm_bound}
\Prb(H\ge \Exp(H)) \,\ge\, \Prb(X\ge \Exp(X)) - \frac{k-1}{n-1} \, .
\end{equation} 
In particular (see Lemma~\ref{prop_main} below), it holds  $\Prb(H\ge \Exp(H))\ge k/n$, when 
$n\ge 8k$ and $\text{Var}(X)\ge 1$. As a consequence, if $i$ is such that $i(n-i)\ge \frac{n^2}{k}$, then loading $n-1$ agents with no amount and one agent with the whole amount is a better choice for the smuggler than loading $i$ agents with an amount of $1/i$ and $n-i$ agents with no amount.

In this work, we aim at complementing the above results and at providing evidence that there is space for improvements.  
Our first result implies that, when $n\ge 8k$, loading one agent with the whole amount is always a best choice for the smuggler.

\begin{theorem}\label{main}
Let $n,k$ be positive integers such that $n\ge 8k$, and let $i\in [n]$. If 
$H\sim\text{Hyp}(n,i,k)$,  then 
\[
\Prb\left(H \ge \Exp(H)\right) \,\ge\, \frac{k}{n} \, .
\]
\end{theorem}

For the proof of Theorem~\ref{main} we combine~\eqref{Ehm_bound} with certain properties of the hypergeometric distribution that allow us to handle the case where $\text{Var}(X)<1$.  
Let us remark that the MMS and Aydinian-Blinovski conjectures mentioned above suggest that the assumption $n\ge 8k$ in Theorem~\ref{main} may be replaced by the assumption $n\ge 4k$, but we are unable to prove it. 
Moreover, we obtain the following lower bound. 

\begin{theorem}\label{main2}
Let $n$ be a positive integer and let $i,k\in [n]$. Suppose that  $H\sim\text{Hyp}(n,i,k)$ is such that $\Exp(H)\in [1\,,\,\min\{i,k\}-1]$ and\,  $\frac{(n-i)(n-k)}{n}> 1$. Then \, 
\[
\Prb(H\ge \Exp(H)) \,\ge\,\frac{e^{-1/12}}{4\sqrt{2}} \cdot \sqrt{\frac{n-1}{n}} \cdot\frac{ \sqrt{\text{Var}(H)} }{1 + \sqrt{1+ \frac{n-1}{n-k}\cdot\text{Var}(H)}} \, . 
\]
\end{theorem}

We prove Theorem~\ref{main2} in Section~\ref{main2_proof}.  
The bound provided by Theorem~\ref{main2} is, of course, weaker than the bound provided by~\eqref{Berry_Essen_bound} but has the advantage that it is obtained using ``first principles", and is applicable on a slightly  wider range of the parameters. 
In particular, Theorem~\ref{main2} is proven by combining a lower bound on the mean absolute deviation (Theorem~\ref{ratio_hyp_bin_exp}) with an upper bound on the tail conditional expectation (Theorem~\ref{TCE_sqrt_bound}) of the hypergeometric distribution, which appear to be new and  may be of independent interest. 

Moreover, Theorem~\ref{main2} occasionally  complements and even improves upon the bound provided in~\eqref{Berry_Essen_bound}. For instance, it follows from~\eqref{Berry_Essen_bound} that  $\Prb(H\ge \Exp(H))\ge 0.01$, when $\text{Var}(H) \ge 1.29$. 
On the other hand, Theorem~\ref{main2} yields the following.

\begin{corollary}\label{cor_bound1}
Let $n\ge 4$ and $H\sim\text{Hyp}(n,i,k)$ be such that $\Exp(H)\in [1, \min\{i,k\}-1]$ and $\frac{(n-i)(n-k)}{n}>1$.  Assume further that $\text{Var}(H)\ge 1$.  Then
$\Prb(H\ge \Exp(H)) \ge 0.049$. 
\end{corollary}

Note that the bound provided by Theorem~\ref{main2} can never be larger than $\frac{e^{-1/12}}{4\sqrt{2}}\approx 0.162$. 
One way to strengthen this bound is by improving upon the constant $\frac{e^{-1/12}}{4\sqrt{2}}$. 
Our numerical experiments suggest that this constant may be replaced by $\frac{1}{2\sqrt{2}}$, but we are unable to provide a proof. Furthermore, our numerical experiments suggest that under the assumptions $\frac{ik}{n}\geq\frac{1}{2}$ and $\frac{(n-i)(n-k)}{n}\geq\frac{1}{2}$  it holds  
$\Prb(H\ge \Exp(H)) \,\ge\, \frac{1}{4}$.

\section{Auxiliary results}
\label{sec:1}

In this section we collect some results regarding  hypergeometric and binomial random variables that will be needed in the proofs of our main results. 

\subsection{Hypergeometric distribution}
\label{hyp_dist1}

We begin with a result that provides a ``refined Markov inequality" for  non-negative integer-valued random variables  whose mean is less than or equal to $1$.

\begin{lemma}
\label{mean_less_than_one}
Suppose that $X$ is a random variable, taking values in the set $\{0,1,2,\ldots\}$, which satisfies $\mathbb{E}(X)\le 1$. Then
\[
\Prb(X=0) \ge \Prb(X\ge 2) \, . 
\]
\end{lemma}
\begin{proof}
Observe that 
\[
\Prb(X=0) + \Prb(X\ge 1) = 1 \ge \Exp(X) \ge \Prb(X\ge 1) + \Prb(X\ge 2) \, .
\] 
\end{proof}

Now we proceed with establishing some properties of the hypergeometric distribution. Our first lemma is a consequence of Lemma~\ref{mean_less_than_one}, and will be needed in the proof of Theorem~\ref{main} to analyze the case where the variance is small. 

\begin{lemma}\label{mean_less_than_one2}
Let $H\sim\text{Hyp}(n,i,k)$ be such that\, $\Exp(H)\le 1$. 
Then 
\[
\mathbb{P}(H=1) \ge \mathbb{P}(H\ge 2) \, .
\]
\end{lemma}
\begin{proof}
Observe that the desired inequality can be equivalently written as 
\[
\mathbb{P}(H=0) + 2\cdot\mathbb{P}(H=1)\ge 1 \, .
\]
Now, we may write 
\[
\mathbb{P}(H=0) + 2\cdot\mathbb{P}(H=1) \,=\, \frac{\binom{n-i}{k}}{\binom{n}{k}} + \frac{2i\binom{n-i}{k-1}}{\binom{n}{k}} 
\,=\,  \frac{\binom{n-i}{k}}{\binom{n}{k}}\cdot\left(1+ \frac{2ki}{n-i-k+1}\right)\, .
\]
We distinguish two cases. Suppose first that $i>1$. 
Let $i,k$ be fixed, and consider the sequence 
\[
q(n) := \frac{\binom{n-i}{k}}{\binom{n}{k}}\cdot\left(1+ \frac{2ki}{n-i-k+1}\right)\, ,\, \text{ for } \, n \ge ik \, .
\]
It is straightforward to verify  that the sequence $\{q(n)\}_{n\ge ik}$ 
is increasing when $n\le 2ik-i-k$,  
and  decreasing when $n\ge 2ik-i-k$. 
Since in the limit it holds $\lim_{n\to\infty}q(n)=1$, it is enough to show that $q(ik)\ge 1$ or, equivalently, that the result holds true when $H\sim\text{Hyp}(ik,i,k)$; i.e., when $\mathbb{E}(H)=1$. In this case, it holds $n=ik$ and we  compute  
\[
\frac{\mathbb{P}(H=1)}{\mathbb{P}(H=0)} = \frac{ik}{ik-i-k+1} =\frac{ik}{(i-1)(k-1)} > 1 \, .
\]
Now, Lemma~\ref{mean_less_than_one} implies that 
\[
\mathbb{P}(H=0) \ge \mathbb{P}(H\ge 2) \, ,
\]
and the result follows. 
Now suppose that $i=1$. Then 
$\Prb(H=0)+ 2\cdot\Prb(H=1) = \frac{n-k}{n} + 2\cdot \frac{k}{n} \ge 1$, 
as desired. 
\end{proof}

The following result allows to express the tail of a hypergeometric random variable in terms of  an ``inverse factorial moment". 

\begin{lemma}\label{clean_form}
Let $H\sim\text{Hyp}(n,i,k)$, 
and fix an integer $m\in \{0,1,\ldots,\min\{i,k\}\}$. 
Then 
\begin{equation}\label{clean_form1}
\mathbb{P}(H \ge m) = \frac{\binom{k}{m}\cdot \binom{i}{m}}{\binom{n}{m}} \cdot \mathbb{E}\left(\frac{1}{\binom{Z_m+m}{m}} \right) \, ,
\end{equation}
where $Z_m\sim\text{Hyp}(n-m,i-m,k-m)$. 
\end{lemma}
\begin{proof}
Observe that 
$Z_m\in\{0,1,\ldots,\min\{i,k\}-m\}$ and that  for $j \in \{m,m+1,\ldots, \min\{i,k\}\}$ we may write  
\begin{eqnarray*}
\mathbb{P}(H = j) &=& \frac{\binom{i}{j}\binom{n-i}{k-j}}{\binom{n}{k}} 
\,=\, \frac{\frac{i}{j}\cdot\frac{i-1}{j-1}\cdots\frac{i-m+1}{j-m+1} \cdot \binom{i-m}{j-m}\binom{n-i}{k-j}}{\frac{n}{k}\cdot \frac{n-1}{k-1}\cdots \frac{n-m+1}{k-m+1}\cdot\binom{n-m}{k-m}} \\
&=& \frac{\binom{k}{m}}{\binom{n}{m}} \cdot\frac{\binom{i}{m}}{\binom{j}{m}} \cdot \mathbb{P}(Z_m = j-m) \, .
\end{eqnarray*}
Hence, it holds that   
\begin{eqnarray*}
\mathbb{P}(H \ge m) &=& \sum_{j=m}^{\min\{i,k\}} \mathbb{P}(H = j) 
\,=\, \sum_{j=m}^{\min\{i,k\}} \frac{\binom{k}{m}}{\binom{n}{m}} \cdot\frac{\binom{i}{m}}{\binom{j}{m}} \cdot \mathbb{P}(Z_m = j-m) \\
&=& \frac{\binom{k}{m} \cdot \binom{i}{m}}{\binom{n}{m}}  \cdot \sum_{j=0}^{\min\{i,k\}-m} \frac{1}{\binom{j+m}{m}} \cdot \mathbb{P}(Z_m = j) 
\,=\, \frac{\binom{k}{m} \cdot \binom{i}{m}}{\binom{n}{m}} \cdot \mathbb{E}\left(\frac{1}{\binom{Z_m+m}{m}}\right) \, ,
\end{eqnarray*}
as desired. 
\end{proof}

\subsection{Mean absolute deviation}
\label{hyp_dist2}

We now proceed with a lower bound on the mean absolute deviation of a hypergeometric random variable. We begin with a result that compares the probabilities of a hypergeometric and a binomial random variable taking the value that is right above their mean.  Here and later, given a real number $x$, we let $\lceil x \rceil=\min\{s\in\mathbb{Z}\,:\, s\ge x\}$. 
The proof employs Robbins' version of Stirling's formula (see~\cite{Robbins}), which states that for every positive integer $n$ it holds 
\begin{equation}\label{robbins_form}
\sqrt{2\pi n} \left(\frac{n}{e}\right)^n e^{\frac{1}{12n+1}} < n! < \sqrt{2\pi n} \left(\frac{n}{e}\right)^n e^{\frac{1}{12n}} \, .
\end{equation}

\begin{lemma}\label{ratio_hyp_bin}
Let $n$ be a positive integer, and let $i,k\in [n]$. 
Let also $H\sim\text{Hyp}(n,i,k)$\,,  $X\sim\text{Bin}(k,\frac{i}{n})$, and 
assume further that the following hold true:
\begin{enumerate} 
\item \, $\frac{ik}{n}\in [1, \min\{i,k\}-2]$. 
\item \, $\frac{(n-i)(n-k)}{n}> 1$. 
\end{enumerate}
If $m=\lceil \frac{ik}{n}\rceil$, then  
\[
\frac{\Prb(H=m)}{\Prb(X=m)} \,\ge\, \frac{e^{-1/12}}{2} \cdot \sqrt{\frac{i(n-i)(n-k)}{(i-m)(n-i-k+m)n}} \, . 
\]
\end{lemma}
\begin{proof}
To simplify notation,  set $s:= n-i-k+m$. 
Note that the assumption $\frac{(n-i)(n-k)}{n}>  1$ implies that $i,k < n-1$. Also, 
note that $s$ is an integer such that 
$s \ge n-i-k+\frac{ik}{n} = \frac{(n-i)(n-k)}{n}$, 
and  $s = \left\lceil \frac{(n-i)(n-k)}{n} \right\rceil$.  
Hence, the assumption $\frac{(n-i)(n-k)}{n}> 1$ implies that $s\ge 2$. 
%Also, note that it holds $s \ge \frac{(n-i)(n-k)}{n} \ge s-1>0$. 
%Finally, note that    
%\begin{equation}\label{m2}
%\frac{i(n-k)}{(i-m)n}\ge 1 \quad\quad  \text{ and } \quad\quad %\frac{(n-i)(n-k)}{sn} \ge 1 - \frac{1}{s}\, .
%\end{equation}
Now we may write  
\[
\frac{\Prb(H=m)}{\Prb(X=m)} = \frac{\binom{i}{m}\binom{n-i}{k-m} n^k}{\binom{n}{k}\binom{k}{m} i^m(n-i)^{k-m}} = 
\frac{i! \cdot (n-i)! \cdot (n-k)! \cdot n^k}{(i-m)!\cdot s! \cdot n! \cdot i^m \cdot (n-i)^{k-m}} \, .
\]
Using~\eqref{robbins_form}, we estimate 
\[
\frac{\Prb(H=m)}{\Prb(X=m)} \ge e^{C_1} \cdot \sqrt{\frac{i(n-i)(n-k)}{(i-m)s n}} \cdot C_2 \, ,
\]
where 
\[
C_1 = \frac{1}{12i+1} + \frac{1}{12(n-i)+1}+\frac{1}{12(n-k)+1} - \frac{1}{12(i-m)} - \frac{1}{12s}-\frac{1}{12n} \, ,
\]
and 
\[
C_2 = \frac{i^i\cdot (n-i)^{n-i}\cdot (n-k)^{n-k}\cdot  n^k}{(i-m)^{i-m}\cdot s^{s}\cdot n^n\cdot i^m\cdot (n-i)^{k-m}}\\
= \left(\frac{i(n-k)}{(i-m)n} \right)^{i-m} \cdot \left( \frac{(n-k)(n-i)}{s n}\right)^{s} \, . 
\]
We  proceed with estimating $C_1$ and $C_2$ from below. 
Since $\frac{1}{12i+1} \ge \frac{1}{12n}$, as well as  $s \ge 2$ and $i-m\ge 2$, it follows that  
$C_1\ge - \frac{1}{12(i-m)} - \frac{1}{12s}\ge -\frac{1}{24}-\frac{1}{24} = -\frac{1}{12}$; hence $e^{C_1} \ge e^{-1/12}$. 

We now show  that $C_2\ge 1/2$. 
Note that if $\frac{ik}{n}$ is an integer then    $\frac{ik}{n}=m$ 
as well as $s = n-i-k-\frac{ik}{n}$, and thus it holds  $C_2=1$; hence we may  assume that $m-1<\frac{ik}{n}<m$. Denote $\delta= m-\frac{ik}{n}$. Then $\delta\in (0,1)$, and we may write  
\[
C_2 \,=\, \left(1+\frac{\delta}{\frac{i(n-k)}{n}-\delta}\right)^{\frac{i(n-k)}{n}-\delta}\cdot\left(1-\frac{\delta}{\frac{(n-i)(n-k)}{n}+\delta}\right)^{\frac{(n-i)(n-k)}{n}+\delta} \, .
\]
To simplify notation,  set $A:=\frac{i(n-k)}{n}$, \, $B=\frac{(n-i)(n-k)}{n}$ and note that $B=n-k-A$. Also, note that the assumption $\frac{ik}{n}\le i-2$ implies that $A\ge 2$. 
Now consider the function  
\[
f(x)=\left(1+\frac{x}{A-x}\right)^{A-x}\cdot\left(1-\frac{x}{B+x}\right)^{B+x}\, , \, \text{ for } \, x\in [0,1] \, .
\]
Note that it holds $f(\delta)=C_2$. 
Moreover, straightforward calculations show that   
$f'(x)=\left(1+\frac{x}{A-x}\right)^{A-x}\cdot\left(1-\frac{x}{B+x}\right)^{B+x}\cdot \ln\left(\frac{B(A-x)}{A(B+x)}\right)$. 
Since $A\ge 2$, it follows that $A-x>0$ for $x\in [0,1]$. Moreover, it holds $\frac{B(A-x)}{A(B+x)} \le 1$, and thus $f'(x)\le 0$. 
This implies that $f$ is decreasing in $[0,1]$; hence we have that  
\[
C_2 \,=\, f(\delta) \,\ge\, f(1) \,=\, \left(1 +\frac{1}{A-1}\right)^{A-1}\cdot\left(1 - \frac{1}{n-k-A+1}\right)^{n-k-A+1} \, . 
\]
Since $A\ge 2$,\, $n-k-A +1 = \frac{(n-i)(n-k)}{n}+1\ge 2$, and the sequences $\{\left( 1+\frac{1}{r}\right)^r\}_{r\ge 1}$ and $\{\left(1-\frac{1}{r}\right)^r\}_{r\ge 1}$ are increasing,  
we conclude that $f(1) \ge 2\cdot \frac{1}{4} =\frac{1}{2}$. The result follows. 
\end{proof}

Our lower bound on the mean absolute deviation of the hypergeometric distribution combines the previous lemma with the corresponding bound for the binomial distribution. 

\begin{theorem}[Berend, Kontorovich~\cite{Berend_Kontorovich}]
\label{BK_Lemma}
Let $k\ge 2$ be a positive integer, and let $X\sim \text{Bin}(k,p)$. If $p\in [\frac{1}{k}, 1- \frac{1}{k}]$, then 
\[
\Exp(|X-kp|) \ge \frac{1}{\sqrt{2}}\cdot\sqrt{\text{Var}(X)} \, .
\]
\end{theorem}

\begin{theorem}\label{ratio_hyp_bin_exp}
Let $n\ge 3$ be a positive integer, and fix $i,k\in [n-1]$.
Let $H\sim\text{Hyp}(n,i,k)$ be such that 
$\Exp(H) \in [1, \min\{i,k\} -1]$ and $\frac{(n-i)(n-k)}{n}> 1$. 
Then \, 
\[
\Exp(|H-ik/n|) \ge \frac{e^{-1/12}}{2\sqrt{2}}  \cdot \sqrt{\frac{n-1 }{n}} \cdot \sqrt{\text{Var}(H)} \, .
\]
\end{theorem}
\begin{proof}
Observe that the assumptions imply that $i,k,n-i,n-k\ge 2$. 
We first suppose that $\Exp(H) \in [1, \min\{i,k\} -2]$, and we  distinguish two cases. 
Assume first that $\frac{ik}{n}$ is not an integer. 
Then we have    
$\frac{ik}{n}\in (m-1, m)$, for some  integer 
$m\in \{2,\ldots, \min\{i,k\}-2\}$, and $m=\lceil \frac{ik}{n}\rceil$. 
Let $X\sim\text{Bin}(k,\frac{i}{n})$ be a binomial random variable having the same mean as $H$, and note that $\frac{1}{k}\le \frac{i}{n}\le 1 - \frac{1}{k}$; in particular,  Theorem~\ref{BK_Lemma} applies to $X$.   
We will need the following expressions for the mean absolute deviations of $H$ and $X$. 
It is known (see~\cite[Formula~(3.3), p.554]{Ramasubban}) that the mean absolute deviation of $H$ can be expressed as 
\begin{equation}\label{ramasuban_formula}
\mathcal{M}_H :=\Exp(|H-ik/n|) = \frac{2m}{n}\cdot (n-i-k+m)\cdot\Prb(H=m) \, .
\end{equation}
Moreover, it is known (see~\cite[Formula~(1.1)]{Diaconis_Zabell}) that the mean absolute deviation of $X$ can be expressed as  
\[
\mathcal{M}_X :=\Exp(|X-ik/n|) = 2m\cdot\frac{n-i}{n}\cdot \Prb(X=m) \, .
\]
Hence, Lemma~\ref{ratio_hyp_bin} implies that  
$\frac{\mathcal{M}_H}{\mathcal{M}_X} \ge  \frac{e^{-1/12}}{2} \cdot \frac{n-i-k+m}{n-i} \cdot \sqrt{\frac{i(n-i)(n-k)}{(i-m)(n-i-k+m)n}}$ 
and therefore, using Theorem~\ref{BK_Lemma}, we estimate  
\begin{eqnarray*}
\mathcal{M}_H &\ge&  \frac{e^{-1/12}}{2} \cdot \frac{n-i-k+m}{n-i} \cdot \sqrt{\frac{i(n-i)(n-k)}{(i-m)(n-i-k+m)n}}\cdot \mathcal{M}_X \\
&\ge&  \frac{e^{-1/12}}{2\sqrt{2}} \cdot \frac{\sqrt{n-i-k+m}}{n-i} \cdot \sqrt{\frac{i(n-i)(n-k)}{(i-m)n}}\cdot \sqrt{\frac{ki(n-i)}{n^2}} \\
&\ge&  \frac{e^{-1/12}}{2\sqrt{2}} \cdot \sqrt{\text{Var}(H)} \cdot \sqrt{\frac{(n-k) i (n-1)}{n^2(i-m)}} \, ,
\end{eqnarray*}
where in the last inequality we have used the fact that $n-i-k+m \ge n-i-k+\frac{ik}{n} = \frac{(n-i)(n-k)}{n}$. 
The result follows upon observing that the assumption $\frac{ik}{n}\le m$ implies that $\frac{(n-k)i}{n(i-m)}\ge 1$. 

Now suppose that $\frac{ik}{n}$ is an integer. In this case the expression for the mean absolute deviation for the hypergeometric distribution is  the same as above, but the corresponding expression for the binomial distribution 
(see~\cite[Formula~(1.1)]{Diaconis_Zabell}) becomes    
$\Exp(|X-ik/n|) = 2(m+1)\cdot\frac{n-i}{n}\cdot \Prb(X=m+1)$. 
Since $m\cdot \Prb(X=m)= (m+1)\cdot \Prb(X=m+1)$, when $m=\frac{ik}{n}$, the result follows as in the previous case. 

We now consider the case $\Exp(H)\in (\min\{i,k\}-2,\min\{i,k\}-1]$. Suppose first that $\min\{i,k\}=i$, thus $\frac{ik}{n}\in (i-2,i-1]$. Now set $Y= i-H$, and notice that $Y\sim\text{Hyp}(n,i,n-k)$,  $\Exp(Y)=\frac{i(n-k)}{n}\in [1,2)$ and  
$\Exp(|H-\frac{ik}{n}|) = \Exp(|Y-\frac{i(n-k)}{n}|)$. 
Moreover, note that $\frac{(n-i)k}{n}\ge \frac{i(n-k)}{n} \ge 1$, since $i\le k$. If $\frac{(n-i)k}{n}=1$, then we also have $\frac{i(n-k)}{n}=1$ and $i=k$, and so it holds $\frac{i(n-i)}{n}=1$. The only solution to $i(n-i)=n$ is  $n=4,i=2$, an instance for which it is easy to see that the desired result holds true. 
So we may suppose that $\frac{(n-i)k}{n} >1$. 
We  distinguish two cases.

Suppose first that $n-k\le i$. 
We claim that the case $n-k=2$ cannot occur under the assumptions of the theorem. Indeed, if $n-k=2$, then the assumption $\Exp(Y)\in [1,2)$ implies that 
$n/2\le i <n$, while the assumption $\frac{2(n-i)}{n}>1$ implies that $i<n/2$. 
Hence we may assume that $n-k\ge 3$. 
If $n-k\ge 4$, then 
it holds $1\le \Exp(Y)\le \min\{i,n-k\}-2$, 
and the result follows from the previous case. The case $n-k=3$ is left to the reader, and the  case $i< n-k$ is obtained along similar reasoning.  

We are left with the case where 
$\Exp(H)\in (\min\{i,k\}-2,\min\{i,k\}-1]$ and $\min\{i,k\}=k$. 
In this case we set $Z = k-H \sim \text{Hyp}(n,n-i,k)$ and we proceed in a similar manner as in the previous case, by the symmetry between the parameters $i$ and $k$.   
\end{proof}

\subsection{Tail conditional expectation}
\label{hyp_dist3} 

In this section we deduce an upper bound on the tail conditional expectation of the hypergeometric distribution. 
In order to state our results, we will need  some basic facts from the theory of stochastic orders, which we briefly discuss here and refer the reader to 
~\cite{Shaked_Shantikumar}, or~\cite{Klenke_Mattner}, for further details and references. 
A random variable $X$ is said to be smaller than the random variable $Y$ in the \emph{usual stochastic order}, denoted $X\le_{st} Y$, if 
it holds 
\[
\Prb(X\ge t) \le \Prb(Y\ge t) \, ,\, \text{ for all } \, t\in \mathbb{R} \, .
\]
Clearly, if $X\le_{st} Y$ then it holds $\Exp(X)\le \Exp(Y)$. 
A random variable $X$  is said to be smaller than the random variable $Y$ in the \emph{likelihood ratio order}, denoted $X\le_{lr} Y$, if the sequence $a_m = \frac{\Prb(X=m)}{\Prb(Y=m)}$ is decreasing in $m$ over the union of the supports of $X$ and $Y$. It is known (see~\cite[Theorem~1.C.1]{Shaked_Shantikumar}, or~\cite{Klenke_Mattner}) that 
$X\le_{lr} Y$ implies $X\le_{st} Y$.

We begin with an upper bound on the tail conditional expectation of random variables whose mean is also a median. 

\begin{lemma}\label{tce_median_bound}
Let $X$ be a random variable whose mean $\mu:=\Exp(X)$ is also a median. 
Then 
\[
\Exp(X\mid X\ge \mu) \le \mu + \sqrt{\text{Var}(X)} \, .
\]
\end{lemma}
\begin{proof} 
Note that the assumption that $\mu$ is a median yields   
\[
\Exp(\max\{0,X-\mu\}) = \Prb(X\ge \mu) \cdot \Exp(X-\mu\mid X\ge \mu) \ge \frac{1}{2} \cdot \Exp(X-\mu\mid X\ge \mu) \, .
\]
Hence it holds  
\[
 \Exp(X-\mu\mid X\ge \mu) \le 2\cdot \Exp(\max\{0,X-\mu\}) \, .
\]
Now, set $M^+ = \max\{0,X-\mu\}$ and $M^- = \max\{0, \mu-X\}$ and note that 
\[
M^+ - M^- = X-\mu \quad\quad\text{ and }\quad\quad 
M^+ + M^- = |X-\mu| \, .
\]
Therefore, we have $\Exp(M^+) = \Exp(M^-)$ as well as\, $2\cdot \Exp(M^+) =\Exp(|X-\mu|)$, and we conclude that  
\[
 \Exp(X-\mu\mid X\ge \mu) \le 2\cdot \Exp(M^+) = \Exp(|X-\mu|) \le \sqrt{\Exp((X-\mu)^2)} = \sqrt{\text{Var}(X)} \, ,
\]
as desired. 
\end{proof}

Our next result asserts that the tail conditional expectation of the hypergeometric distribution is dominated by the tail conditional expectation of the binomial distribution; thus an upper bound on the latter is also an upper bound on the former.

\begin{lemma}\label{tce_conj}
Let $H\sim\text{Hyp}(n,i,k)$ and $X\sim\text{Bin}(k, \frac{i}{n})$.
Let $H^{\ast}$ denote the  random variable $H$ conditional on the event $\{H\ge ik/n\}$. Similarly, let 
$X^{\ast}$ denote the  random variable $X$ conditional on the event $\{X\ge ik/n\}$.
Then   
\[
H^{\ast}\, \le_{st} \, X^{\ast} \, .
\]
In particular, it holds 
\[
\Exp(H\mid H\ge \Exp(H)) \le \Exp(X\mid X\ge \Exp(X) ) \, .
\]
\end{lemma}
\begin{proof}
Observe that it is enough to show that $H^{\ast} \le_{lr}  X^{\ast}$, and  that it holds 
\[
H^{\ast} \in \left\{\left\lceil\frac{ik}{n} \right\rceil,\ldots, \min\{i,k\}\right\} 
\quad\quad\text{ and }\quad\quad X^{\ast} \in \left\{\left\lceil\frac{ik}{n} \right\rceil,\ldots, k\right\} \, .
\]
Let $a_m = \frac{\Prb\left(H^{\ast}=m \right)}{\Prb\left(X^{\ast}=m\right)}$, for $m\in  \{\lceil\frac{ik}{n} \rceil,\ldots,k\}$, and note that $a_m=0$ when $m > \min\{i,k\}$. Hence, we have  to show that the sequence $\{a_m\}_{m=\lceil ik/n\rceil}^{\min\{i,k\}}$ is decreasing.  
Since 
\[
a_m = \frac{\Prb\left(H^{\ast}=m \right)}{\Prb\left(X^{\ast}=m\right)} = \frac{\Prb(H=m\mid H\ge ik/n)}{\Prb(X=m\mid X\ge ik/n)} = \frac{\Prb(H=m)\cdot \Prb(X\ge ik/n)}{\Prb(X=m)\cdot \Prb(H\ge ik/n)} \, ,
\]
it is straightforward to verify that it holds   
\[
\frac{a_m}{a_{m+1}} = \frac{i(n-i-k+m+1)}{(n-i)(i-m)} \, .
\]
The result follows upon observing that $\frac{a_m}{a_{m+1}}\ge 1$ when $m\ge \frac{i(k-1)}{n}$. 
\end{proof}

Moreover, we will need the following comparison between the tail conditional expectations of binomial random variables.

\begin{lemma}\label{tce_binomial}
Let $k$ be a positive integer, and $p,q\in (0,1)$ be such that $p\le q$. Let $X_p\sim\text{Bin}(k,p)$ ,   $X_q\sim\text{Bin}(k,q)$ 
and fix a positive integer $m\in \{0,1,\ldots,k\}$. Then 
\[
\Exp(X_p\mid X_p \ge m) \,\le\, \Exp(X_q\mid X_q \ge m) \, .
\]
\end{lemma}
\begin{proof} 
This is a well-known result (see~\cite{Broman_Brug_Kager_Meester}, \cite{Klenke_Mattner}, or~\cite[Lemma~2.4]{Pelekis_Ramon}). 
\end{proof}

We can now state our upper bound on the tail conditional expectation of the hypergeometric distribution. Recall that $\lceil x \rceil=\min\{s\in\mathbb{Z}\,:\, s\ge x\}$.

\begin{theorem}\label{TCE_sqrt_bound}
Let $H\sim\text{Hyp}(n,i,k)$. Then 
\[
\Exp(H\mid H\ge \Exp(H)) \,\,\le \,\, \lceil\Exp(H)\rceil  + \sqrt{\text{Var}(H) \cdot \frac{n-1}{n-k} + 1} \, .
\]
\end{theorem}
\begin{proof}
From Lemma~\ref{tce_conj} we know that 
\begin{equation}\label{tce_b1}
\Exp(H \mid H\ge \Exp(H)) \le \Exp(X\mid X\ge \Exp(X)) \, ,
\end{equation}
where $X\sim\text{Bin}(k, \frac{i}{n})$.  Let  $m := \lceil ik/n\rceil$, and consider a random variable 
$Y\sim\text{Bin}(k, \frac{m}{k})$. Since $\frac{i}{n} \le \frac{m}{k}$, it follows from Lemma~\ref{tce_binomial} that 
\begin{equation}\label{tce_b2}
\Exp(X \mid X\ge \Exp(X)) \,\le\, \Exp(Y \mid Y\ge \Exp(X)) \,=\, \Exp(Y \mid Y\ge \Exp(Y)) \, ,
\end{equation}
where the last equality follows from the fact that $\Exp(Y) = m=\lceil \Exp(X)\rceil $. 
Since $Y$ is a binomial random variable whose mean is an integer, it is well-known (see~\cite{Jogdeo_Samuels} or~\cite[Theorem~2.2]{Siegel})
that $m$ is also a median of $Y$. Hence Lemma~\ref{tce_median_bound} implies that 
\begin{equation}\label{tce_b3}
\Exp(Y \mid Y\ge \Exp(Y)) \le m + \sqrt{\text{Var}(Y)} \, . 
\end{equation}
Summarizing the above, it follows from~\eqref{tce_b1}, \eqref{tce_b2} and \eqref{tce_b3} that 
\begin{eqnarray*}
\Exp(H \mid H\ge \Exp(H)) &\le& m + \sqrt{k\cdot \frac{m}{k}\cdot\frac{k-m}{k}} 
\,=\, m + \sqrt{m \cdot\frac{k-m}{k}} \\
&\le&  m + \sqrt{\left(\frac{ik}{n} +1\right) \cdot\frac{k-ik/n}{k}} 
\,=\,  m + \sqrt{\left(\frac{ik}{n} +1\right) \cdot\frac{n-i}{n}}  \\
&\le& m + \sqrt{k\cdot \frac{i}{n}\cdot \frac{n-i}{n} + 1}
\,=\, m + \sqrt{\text{Var}(H) \cdot \frac{n-1}{n-k} + 1} \, ,
\end{eqnarray*}
as desired. 
\end{proof}

\section{Proof of Theorem~\ref{main}}
\label{main_proof}

Throughout this section, we assume that 
$n,k$ are positive integers such that $n\ge 8k$, and  $i\in [n]$, all the parameters being fixed.
Moreover, we fix a hypergeometric random variable  $H\sim\text{Hyp}(n,i,k)$ and a binomial random variable   $X\sim\text{Bin}(k,\frac{i}{n})$  having the same mean as $H$. Our proof employs  the following  result of Ehm (see~\cite[Theorem~$2$]{Ehm}). 
Recall that the \emph{total variation distance} between $H$ and $X$ is defined as 
\[
d_{TV}(H \, , \, X) = \sup_{A\subseteq \mathbb{N}} \, |\Prb(H\in A) - \Prb(X\in A)| \, . 
\]

\begin{theorem}[Ehm~\cite{Ehm}]
\label{Ehm}
    Suppose that the parameters $n,i,k$ are such that\, $k\cdot\frac{i}{n}\cdot \frac{n-i}{n}\ge 1$. Then 
    \[
    d_{TV}(H \, , \, X) \le \frac{k-1}{n-1} \, . 
    \]
\end{theorem}
\noindent
We will also need the following constant lower bound on the probability that a binomial random variable exceeds its mean. 

\begin{theorem}[Greenberg, Mohri~\cite{greenberg2014tight}]
\label{Greenberg_Mohri}
Let $k$ be a positive integer and let $p\in (\frac{1}{k},1)$. If  $Y\sim\text{Bin}(k,p)$, then 
\[
\mathbb{P}(Y \ge \Exp(Y)) \, \ge \, \frac{1}{4} \, .
\]
\end{theorem}

To prove Theorem~\ref{main}, we will distinguish four cases. 
We begin with the case where $k\cdot\frac{i}{n}\cdot \frac{n-i}{n}\ge 1$, which follows readily  from the above mentioned theorems.

\begin{lemma}\label{prop_main}
Suppose that $n,i,k$ are such that\, $k\cdot\frac{i}{n}\cdot \frac{n-i}{n}\ge 1$ and $n\ge 8k$. 
Then 
\[
\mathbb{P}(H \ge ik/n) \,\ge\, k/n \, .
\]
\end{lemma}
\begin{proof}
From Theorem~\ref{Ehm} we know that 
\[
\mathbb{P}(H \ge ik/n) \ge \mathbb{P}(X \ge ik/n) - \frac{k-1}{n-1}
\]
and is therefore enough to show that 
$\mathbb{P}(X \ge ik/n)  \ge \frac{k}{n} + \frac{k-1}{n-1}$. 
Since $k\cdot\frac{i}{n}\cdot \frac{n-i}{n}\ge 1$,  it also holds 
$k\cdot\frac{i}{n}> 1$, and 
Theorem~\ref{Greenberg_Mohri} implies that 
$\mathbb{P}(X \ge ik/n) \ge \frac{1}{4}$.
Hence  it is enough to show that 
$\frac{1}{4} \ge \frac{k}{n} + \frac{k-1}{n-1}$, 
which is clearly correct when $n\ge 8k$. 
\end{proof}

It follows from Lemma~\ref{prop_main} that we may suppose for the rest of this section that the parameters $n,i,k$ satisfy 
\begin{equation}\label{main_assumption1}
k\cdot\frac{i}{n}\cdot \frac{n-i}{n}< 1 \quad \text{ and } \quad n\ge 8k \, .
\end{equation}
This leaves three cases to consider, depending on whether $i$ is smaller than, equal to, or larger than $n/2$. 
Those three cases above are considered in the following three lemmata, whose proofs complete the proof of Theorem~\ref{main}. 
We begin with the first case. 

\begin{lemma}\label{mms1}
Suppose that~\eqref{main_assumption1} holds true, and that  $i<n/2$.  
Then  
\[
\mathbb{P}(H \ge ik/n) \, \ge \, \frac{k}{n} \, .
\]
\end{lemma}
\begin{proof}
Since $i<n/2$, it follows from~\eqref{main_assumption1}  that $\frac{ik}{n} < \frac{n}{n-i} < 2$, and thus $\frac{ik}{n}\in (0,2)$. 
We consider two cases. 

Suppose first that $\frac{ik}{n} \in (0,1]$.  
We may assume that $i\ge 2$; otherwise the result is clearly correct. 
Since $\frac{ik}{n} \le 1$ it holds $\mathbb{P}(H \ge ik/n) =\mathbb{P}(H \ge 1)$ 
and therefore it is enough to show that
\[
\mathbb{P}(H\ge 1)\ge \frac{k}{n} \, .
\] 
Now, Lemma~\ref{clean_form} implies that 
\[
\mathbb{P}(H\ge 1)=\frac{k}{n} \cdot \mathbb{E}\left(\frac{i}{Z_1 + 1}\right)   \, , 
\]
where $Z_1\sim\text{Hyp}(n-1,i-1,k-1)$.   The result follows upon observing that $i\ge Z_1+1$. 

Suppose now that $\frac{ik}{n} \in (1,2)$ and $i<\frac{n}{2}$. 
Observe that the assumption $n\ge 8k$ implies that $i> \frac{n}{k} \ge 8$, and so we may assume that $i\ge 9$. 
We may also assume that $k\ge 2$; otherwise the result is clearly true. 
Since $1<\frac{ik}{n} < 2$ it follows that 
$\mathbb{P}(H \ge \frac{ik}{n})=\mathbb{P}(H\ge 2)$. 
Moreover, it follows from Lemma~\ref{clean_form} that 
\[
\mathbb{P}(H\ge 2) = \frac{k(k-1)}{n(n-1)} \cdot i(i-1)\cdot \mathbb{E}\left(\frac{1}{(Z_2+2)(Z_2+1)}\right) \, ,
\]
where $Z_2\sim\text{Hyp}(n-2,i-2,k-2)$. 
Jensen's inequality now implies that 
\[
\mathbb{E}\left(\frac{1}{(Z_2+2)(Z_2+1)}\right) \ge 
\frac{1}{\mathbb{E}((Z_2+2)(Z_2+1))} = \frac{1}{\mathbb{E}(Z_2^2+3Z_2 +2)} \, .
\]
Since $Z_2\sim\text{Hyp}(n-2,i-2,k-2)$, we compute 
\[
\mathbb{E}(Z_2^2+3Z_2 +2) = A \cdot 
\frac{(n-i)(n-k)}{(n-2)(n-3)} + A^2 + 3A + 2 \, , \, \text{ where } \, A := \frac{(i-2)(k-2)}{n-2} \, .
\]
Hence, in order to prove the desired result, it is enough to prove that 
\[
\frac{k(k-1)}{n(n-1)} \cdot i(i-1)\cdot \frac{1}{\mathbb{E}(Z_2^2+3Z_2 +2)} \ge \frac{k}{n} \, ,
\]
which is equivalent to
\begin{equation}\label{ineq_mm1}
A \cdot 
\frac{(n-i)(n-k)}{(n-2)(n-3)} + A^2 + 3A + 2 \le i\cdot \frac{(i-1)(k-1)}{n-1} \, .
\end{equation}
Since $\frac{(n-i)(n-k)}{(n-2)(n-3)}<1$, \, $A\le B:=\frac{(i-1)(k-1)}{n-1}$ and $i\ge 9$, we conclude that~\eqref{ineq_mm1} will follow upon showing that 
\begin{equation}\label{ineq_mm2}
B^2 +4B + 2 \le 9B \quad \Longleftrightarrow\quad B^2-5B+2 \le 0 \, .
\end{equation}
Solving for $B$, we find that~\eqref{ineq_mm2} holds true for $B\in \left[\frac{1}{2}(5-\sqrt{17})\, ,\, \frac{1}{2}(5+\sqrt{17})\right]$; hence, since $B\le \frac{ik}{n}\le 2$, the proof will be complete after showing that $B > \frac{1}{2}(5-\sqrt{17})\approx 0.43844$. 
To this end, observe that the assumptions $i\ge 9$ and $i<\frac{n}{2}$ imply that 
\[
\frac{B}{\frac{ik}{n}}=\frac{\frac{(i-1)(k-1)}{n-1}}{\frac{ik}{n}} = \frac{n}{n-1} \cdot \left(1-\frac{1}{i}\right)\cdot \left(1-\frac{1}{k}\right) > \frac{8}{9}\cdot \left(1-\frac{i}{n}\right) >\frac{8}{9}\cdot\frac{1}{2} > 0.44 \, .
\]
Since  $\frac{ik}{n}>1$, it follows that $B > 0.44 \cdot \frac{ik}{n} > 0.44$,  and the result follows. 
\end{proof}

The next result settles the second case. 

\begin{lemma}\label{third_case}
Suppose that~\eqref{main_assumption1} holds true, and that  $i=n/2$. 
Then   
\[
\mathbb{P}(H \ge ik/n) \, \ge \, \frac{k}{n} \, .
\]
\end{lemma}
\begin{proof}
Since $i=n/2$, it follows that $\frac{ik}{n} =k/2$. 
We claim that in this case it holds $\mathbb{P}(H \ge \frac{k}{2}) \ge \frac{1}{2}$. 
Since $\frac{k}{n}\le \frac{1}{8}$ the result  follows from the claim. 
To prove the claim, observe that for every sample for which the event $\{H\ge k/2\}$ occurs, there is a unique sample (its complement) for which $\{H\le k/2\}$ occurs. This implies that the event $\{H\ge k/2\}$ occurs for at least half of the samples, and the claim follows. 
\end{proof}

Finally, we consider the third case.  

\begin{lemma}
Suppose that~\eqref{main_assumption1} holds true, and 
that $i>n/2$.  
Then  
\[
\mathbb{P}(H \ge ik/n) \, \ge \, \frac{k}{n} \, .
\]  
\end{lemma}
\begin{proof}
Since $i>n/2$, it follows from~\eqref{main_assumption1} that 
$k-\frac{ik}{n} = \frac{k(n-i)}{n} < \frac{n}{i} < 2$, and thus  $\frac{ik}{n} > k-2$. 
Hence it holds $\frac{ik}{n}\in (k-2,k)$. 
We distinguish two cases. 

Suppose first that $\frac{ik}{n}\in (k-1,k]$. 
We have to show that $\mathbb{P}(H=k)\ge \frac{k}{n}$. Since $n\ge 8k$, we may assume that $\mathbb{P}(H=k)<\frac{1}{8}$; otherwise the result holds true. 
We may also assume that $k\ge 2$. 

Set $W = k -H$. Observe that $W\sim\text{Hyp}(n,n-i,k)$ and $\mathbb{E}(W)<1$. 
Now Lemma~\ref{mean_less_than_one2} implies that 
$\mathbb{P}(W=1)\ge \mathbb{P}(W\ge 2)$ and therefore we have  
\begin{equation*}
\mathbb{P}(H= k-1) \ge \mathbb{P}(H \le k-2) \, .
\end{equation*}
Since $\mathbb{P}(H \le k-1) > \frac{7}{8}$, 
upon adding the term $\mathbb{P}(H= k-1)$ on both sides of the previous inequality, we deduce that 
$2\cdot\mathbb{P}(H=k-1) \ge \mathbb{P}(H \le k-1) > \frac{7}{8}$\,; hence  
$\mathbb{P}(H=k-1) \ge\frac{7}{16}$.
Since $\frac{\mathbb{P}(H=k)}{\mathbb{P}(H=k-1)}=\frac{i-k+1}{k(n-i)}$, we conclude that it holds 
\[
\mathbb{P}(H=k) = \frac{i-k+1}{k(n-i)}\cdot \mathbb{P}(H=k-1) > \frac{7}{16}\cdot\frac{i-k+1}{k(n-i)} \, ,
\]
and the result will follow upon showing that the inequality $\frac{7}{16}\cdot\frac{i-k+1}{k(n-i)}\ge \frac{k}{n}$ holds true; this is equivalent to 
\begin{equation*}
7ni + 16k^2 i + 7n \ge 16k^2n + 7kn \, .
\end{equation*}
Since $i > n\cdot (1- \frac{1}{k})$, by assumption, it is enough to show that 
\[
7n^2\left(1- \frac{1}{k}\right) + 16k^2n\left(1- \frac{1}{k}\right) + 7n \ge 16k^2n+7kn \quad \Longleftrightarrow \quad 7n\left(1- \frac{1}{k}\right) +7 \ge 23k \, .
\]
Since $k\ge 2$ and $n\ge 8k$, it is easy to verify that the later inequality holds true. 
The result follows for the case $\frac{ik}{n}\in (k-1,k]$. 

Suppose now that $\frac{ik}{n}\in (k-2,k-1]$. 
We may assume that $k\ge 3$; otherwise, the result follows from Lemma~\ref{mms1}. 
We may also suppose that $\mathbb{P}(H\ge k-1)<\frac{1}{8}$; otherwise, the result follows from the fact that $n\ge 8k$. In other words, we may assume that 
\begin{equation}\label{last1}
\mathbb{P}(H \le k-2) \ge \frac{7}{8} \, .
\end{equation}
Since the difference between a median and the mean of a hypergeometric random variable is at most one (see~\cite{Uhlmann}, \cite{Siegel} or~\cite[Corollary~13]{cooper2009linearly}), it follows that  
\begin{equation}\label{last2}
\mathbb{P}(H \le k-3) < \frac{1}{2} \, .
\end{equation}
If we combine~\eqref{last1} and~\eqref{last2}, we conclude that it holds 
\begin{equation}\label{last3}
\mathbb{P}(H = k-2) \ge \frac{3}{8} \, .
\end{equation}
Now observe that we need to show that $\mathbb{P}(H\ge k-1)\ge \frac{k}{n}$ and so it is enough to show that  
$\mathbb{P}(H = k-1)\ge \frac{k}{n}$. 
Since 
\[
\frac{\mathbb{P}(H=k-1)}{\mathbb{P}(H=k-2)} = \frac{2(i-k+2)}{(k-1)(n-i-1)} 
\]
it follows from~\eqref{last3} that it is enough to show that 
\begin{equation}\label{last4}
\frac{3}{4}\cdot\frac{i-k+2}{(k-1)(n-i-1)} \ge \frac{k}{n} \quad \Longleftrightarrow\quad 3ni+4i(k^2-k) +nk+6n+4k^2 \ge 4k^2n+4k \, .
\end{equation}
Now, the assumptions $i> \frac{n}{2}$ and $i>n\cdot\frac{k-2}{k}$  
imply that 
\[
3ni > \frac{3}{2}n^2 \quad\quad\text{ and }\quad\quad 4i(k^2-k) > 4n\cdot\frac{k-2}{k}\cdot(k^2-k)= 4k^2n-12kn+8n \, ,
\] 
and  therefore~\eqref{last4} will follow upon showing that it holds  
\begin{equation}\label{last5}
 \frac{3}{2}n^2+14n+4k^2 \ge 11kn+4k \, .
\end{equation}
Finally,  the assumption $n\ge k$ implies that  instead of~\eqref{last5}  it is enough to show that 
$\frac{3}{2}n^2+14n \ge 11kn+4n$, which is equivalent to 
\[
n + \frac{20}{3} \ge \frac{22}{3} k \, .
\]
The result follows upon recalling  that we assume $n\ge 8k$. 
\end{proof}

\section{Proof of Theorem~\ref{main2}}
\label{main2_proof}

In this section we prove Theorem~\ref{main2}. 
Observe that in the proof of Lemma~\ref{tce_median_bound} we have used the fact that for any random variable $H$ it holds 
\begin{equation}\label{mad_tce_b1}
\frac{1}{2}\cdot \Exp(|H - \Exp(H)|) = \Prb(H\ge \Exp(H))\cdot \Exp(H - \Exp(H)\mid H\ge \Exp(H)) \, .
\end{equation}

\begin{proof}[Proof of Theorem~\ref{main2}]
Let $H\sim\text{Hyp}(n,i,k)$. Using~\eqref{mad_tce_b1}, we may write   
\[
\Prb(H\ge \Exp(H)) = \frac{1}{2}\cdot \frac{\Exp(|H - \Exp(H)|)}{\Exp(H - \Exp(H)\mid H\ge \Exp(H))} \, .
\]
Now, Theorem~\ref{ratio_hyp_bin_exp} implies that 
\begin{equation}\label{mad_b1}
\Exp(|H - \Exp(H)|) \,\ge\,   
 \frac{e^{-1/8}}{2\sqrt{2}}\cdot \sqrt{\frac{n-1}{n}} \cdot \sqrt{\text{Var}(H)}  \, . 
\end{equation}
Theorem~\ref{TCE_sqrt_bound} implies that 
\begin{eqnarray*}
\Exp(H - \Exp(H)\mid H\ge \Exp(H)) &\le& \lceil\Exp(H)\rceil  + \sqrt{1+\text{Var}(H) \cdot \frac{n-1}{n-k}} - \Exp(H) \\ 
&\le&  1 + \sqrt{1 +\text{Var}(H) \cdot \frac{n-1}{n-k} } \, .
\end{eqnarray*}
Hence, it holds 
\[
\Prb(H\ge \Exp(H)) \ge \frac{e^{-1/8}}{4\sqrt{2}} \cdot \sqrt{\frac{n-1}{n}} \cdot\frac{ \sqrt{\text{Var}(H)} }{1 + \sqrt{1+\text{Var}(H) \cdot \frac{n-1}{n-k}}} \, ,
\]
as desired. 
\end{proof}

We end this section with a proof of Corollary~\ref{cor_bound1}.

\begin{proof}[Proof of Corollary~\ref{cor_bound1}]
Observe that hypotheses imply that $k\ge 2$. 
Suppose first that $k\le n/2$, say  $k =\alpha\cdot n$, for some $\alpha\in (0,1/2]$. 
From Theorem~\ref{main2} we know that 
\[
\Prb(H\ge \Exp(H)) \ge \frac{e^{-1/8}}{4\sqrt{2}} \cdot \sqrt{\frac{n-1}{n}} \cdot\frac{ \sqrt{\text{Var}(H)} }{1 + \sqrt{1+\frac{n-1}{n-k}\cdot \text{Var}(H)}} \, .
\]
Now, it is not difficult to see that for $c>0$ fixed, the function 
\[
f(x)= 1 + \sqrt{1 + c\cdot x} - (1+\sqrt{1+c})\cdot \sqrt{x} 
\]
is decreasing in the interval $[1 , \infty)$. Hence it holds $f(x) \le 0$, for $x\ge 1$. 
Since $\text{Var}(H)\ge 1$,  this implies that  
\[
\frac{ \sqrt{\text{Var}(H)}  }{1 + \sqrt{1+\frac{n-1}{n-k}\cdot \text{Var}(H)}} \ge \frac{\sqrt{\text{Var}(H)}}{ \left(1+ \sqrt{1+\frac{n-1}{n-k}}\right)\cdot \sqrt{\text{Var}(H)} }  \ge \frac{1}{1+ \sqrt{1+\frac{n}{n-k}}}  = \frac{1}{1+\sqrt{1 + \frac{1}{1-\alpha}}}\, .
\]
Since $n\ge 4$, it follows  $\frac{n-1}{n}\ge \frac{3}{4}$ and we conclude that  
\[
\Prb(H\ge \Exp(H)) \ge \frac{e^{-1/8}}{4\sqrt{2}}\cdot \sqrt{\frac{3}{4}} \cdot\frac{1}{1+\sqrt{1+\frac{1}{1-\alpha}}}  \, .
\]
Since the right hand side is decreasing in $\alpha$, the desired result follows upon substituting $\alpha =1/2$. 

Now suppose that $k > n/2$. Let $H_1\sim \text{Hyp}(n,n-i,n-k)$. Then $n-k < n/2$ as well as $1<\frac{(n-i)(n-k)}{n}\le \min\{n-i,n-k\}-2$, and  the result follows upon observing  that 
$\Prb(H\ge \Exp(H)) = \Prb(H_1\ge \Exp(H_1))$.
\end{proof}

\section*{Acknowledgements}
JA was supported by the Czech Science Foundation project 24-11820S (“Automatic Combinatorialist”).

\end{document}